%% file: Spectral_Sullivan.tex
\documentclass[12pt,oneside]{article}
\input{template.tex}
\title{The spectral Sullivan conjecture}
\author{Ishan Levy}
\usepackage{setspace}
\usepackage{babel,csquotes,xpatch}
\usepackage[backend=bibtex,style=alphabetic,maxnames = 99,maxalphanames=99]{biblatex}
\makeatletter
\renewcommand\tableofcontents{%
	\@starttoc{toc}%
}
\makeatother

\newcounter{counter}
\newtheorem{thm}[counter]{Theorem}

\newcommand{\Addresses}{{
		\bigskip
		\footnotesize
		\textsc{Department of Mathematics, Institute of Advanced Studies, USA}\par\nopagebreak
		\textit{E-mail address}: \texttt{ishanl@ias.edu}
}}
\begin{document}
	\date{}
	\maketitle
	\begin{abstract}
		We show that any map from an infinite loop space to a $p$-complete nilpotent finite dimensional space factors canonically through a union of $p$-adic tori. This is proven via bootstrapping from the case of $B\ZZ/p\ZZ$, which is the key case of the Sullivan conjecture proven by Miller. The main step in our proof is to show that the subcategory of spectra generated by the reduced suspension spectrum of $B\ZZ/p\ZZ$ under colimits and extensions agrees with that of a Moore spectrum.
	\end{abstract}
	\begin{spacing}{0.1}
		\tableofcontents
	\end{spacing}
\section{Introduction}

There is a dichotomy in homotopy theory of the types of spaces that are often studied. On one hand, there are spaces that are equivalent to finite dimensional CW-complexes such as many spaces coming from manifolds, varieties. On the other hand, there are spaces that are used as `homotopy theoretic moduli spaces', i.e maps into them classify some kind of homotopy theoretic data associated to the source. Spaces of this form include $BG$ for a compact Lie group $G$, finite Postnikov towers, and $\Omega^{\infty}E$ where $E$ is a spectrum. 

Often one detects information about finite dimensional CW-complexes by mapping into homotopy theoretic moduli spaces. The Sullivan conjecture, proposed by Sullivan in \cite{sullivan1970geometric}, and resolved by Miller \cite{miller1984sullivan} (see also \cite{carlsson1991equivariant}) puts severe restrictions on maps in the other direction $p$-adically.\footnote{Rationally, the space of such maps is often highly nontrivial, because the suspension of any simply connected rational space is a wedge of spheres.}

Let $\Spc$ denote the ($\infty$-)category of spaces. We say that $Y\in \Spc$ is \textit{finite dimensional} if it is the homotopy type of a finite dimensional CW-complex. We use $\FF_p$-localization to refer to the Bousfield localization of spaces with respect to $\FF_p$-homology. The key case of the Sullivan conjecture proven by Miller is the following:

\begin{thm}[{\cite[Theorem C]{miller1984sullivan}}]\label{thm:miller}
	Let $Y$ be a finite dimensional nilpotent\footnote{This means that $\pi_1Y$ is nilpotent at each base point and acts unipotently on the higher homotopy groups of $Y$. In particular, it is satisfied for $Y$ simply connected.} $\FF_p$-local space. Then the map $\Map(BC_p,Y) \to Y$ is an equivalence.
\end{thm}

We note that Miller actually proves an integral statement, but it quickly reduces to the result above.

The Sullivan conjecture was then generalized to show that many other homotopy theoretic moduli spaces also do not admit many maps to such $X$. For example, in \cite{zabrodsky1987phantom}, it was shown that the above result holds when $BC_p$ is replaced by a finite Postnikov towers with torsion homotopy groups, and the same result holds for $BG$ for $G$ a compact Lie group follows from \cite{jackowski1992homotopy}.

The goal of this paper is to prove a generalization of \Cref{thm:miller} where $BC_p$ may be replaced by $\Omega^{\infty} E$ for any spectrum $E$.

Given an abelian group $A$, we use $(A/\mathrm{tors})^{\wedge}_p$ to denote the $p$-completion of the quotient of $A$ by its torsion subgroup, which is always a free $p$-complete abelian group.

\begin{theorem}\label{thm:spectralsullivanintro}
	Let $X$ be an infinite loop space, and consider the map $g:X \to \coprod_{\pi_0X}B((\pi_1X/\mathrm{tors})^{\wedge}_p)$ which is a bijection on components and on each component induces the map $\pi_1X \to (\pi_1X/\mathrm{tors})^{\wedge}_p$ on fundamental groups. Then for any nilpotent finite dimensional $p$-complete space $Y$, $g$ induces an equivalence on mapping spaces into $Y$.
\end{theorem}
%

Since $(\pi_1X/\mathrm{tors})^{\wedge}_p$ is a free $p$-complete abelian group, the target of the map $g$ above can be thought of as a union of $p$-adic tori indexed by $\pi_0X$. Then \Cref{thm:spectralsullivanintro} informally says that any map from an infinite loop space to a nilpotent finite dimensional $p$-complete space factors canonically through a union of $p$-adic tori.

As a consequence of the above theorem, actions of infinite loop spaces on finite dimensional spaces are quite restricted:
\begin{corollary}\label{cor:action}
	Any action of an infinite loop space $X$ on a $p$-complete nilpotent finite dimensional space $Y$ canonically factors through an action of a group that is an extension of a discrete abelian group by a $p$-adic torus.
\end{corollary}

\begin{remark}
	We also prove integral variants of the above results in \Cref{thm:integralspecsull} and \Cref{cor:intaction}.
\end{remark}


To prove \Cref{thm:spectralsullivanintro}, we consider the presentable localization of spaces where we $\FF_p$-localize and also apply the Bousfield--Dror Farjoun nullification with respect to $BC_p$. Let $L$ denote this localization functor. The Sullivan conjecture implies that nilpotent $p$-complete finite spaces are local in this localization. \Cref{thm:spectralsullivanintro} then follows from the statement that the map $g$ agrees with the natural transformation $X \to LX$.

To prove this, we first show using a well known principle of unstable localizations that because the map $BC_p \to *$ is an $L$-equivalence, so is $\Omega^{\infty}\Sigma^{\infty}BC_p \to *$. The key point is to use this to show that $\Omega^{\infty}\Sigma^{\infty}S^1/p \to *$ is an $L$-equivalence. To do this, it is convenenient to use the notion of a connective Bousfield class, which was first introduced by Bousfield \cite[Section 3]{bousfield1996unstable}.

\begin{definition}\label{def:cnbousfieldclass}
	Let $X$ be a spectrum. The \textit{connective Bousfield class} of $X$ is the subcategory of spectra generated under colimits and extensions by $X$. We use $\langle X\rangle^{s}$ to denote the connective Bousfield class of $X$.
\end{definition}

Now the main ingredient to proving that $g$ agrees with the map $X \to LX$ is the following theorem:

\begin{theorem}\label{thm:cnbousfieldclassmain}
	There is an equality of connective Bousfield classes:
	$$\langle \Sigma\SP/p\rangle^{s} = \langle \Sigma^{\infty}BC_p\rangle^{s}$$
\end{theorem}

The above theorem is a strengthening of a result announced by Hopkins and Smith \cite{hopkinssmithCPinfty} showing that the Bousfield classes of $\SP/p$ and $\Sigma^{\infty}BC_p$ agree.

%


We also provide an alternative proof of \Cref{thm:cnbousfieldclassmain}, using the symmetric power filtration on $\ZZ$. This proof suggests the following generalization:

\begin{conjecture}\label{question:steinberg}
	Let $L(k)$ be $\Sigma^{-k}\mathrm{Sp}^{p^k}/\mathrm{Sp}^{p^{k-1}}\SP_{(p)}$, where $\mathrm{Sp}^{p^k}$ is the $p^{k}$th symmetric power. Then the connective Bousfield class of $L(k)$ the same as that of some type $k$ finite spectrum.
\end{conjecture}

%
%

\subsection*{Acknowledgements}
We would like to thank Shaul Barkan, Robert Burklund, Jesper Grodal, Gijs Heuts, Mike Hopkins, and Vignesh Subramanian for helpful conversations related to this paper. We also thank Gijs Heuts for feedback on an earlier draft of the paper. This work was done while the author was supported by the Clay Research Fellowship.

\section{Connective Bousfield classes}



The goal of this section is to prove \Cref{thm:cnbousfieldclassmain}.

The following is a stable analog of a result of Bousfield \cite[Theorem 9.10]{bousfield1994localization}. It says that the connective Bousfield class of $Z$ only depends on the homology of $Z$ and the `stable' connective Bousfield class of $Z$.
\begin{lemma}\label{lemma:cnsuspension}
	Let $Z$ be a connective spectrum. Then for any $n\geq0$, $\langle Z\rangle ^{s} = \langle \Sigma^n Z \oplus Z\otimes \ZZ\rangle^{s}$
\end{lemma}

\begin{proof}
	Clearly $\langle Z\rangle ^{s} \supset \langle \Sigma^n Z \oplus Z\otimes \ZZ\rangle^{s}$. For the other inclusion, it is clear that by induction on $n$, it suffices to prove it for $n=1$. It is enough to show that the fiber of the map $Z \to Z\otimes \ZZ$ is in $\langle \Sigma Z\rangle^{s}$. But this is $Z\otimes \tau_{\geq1}\SP$, and $\tau_{\geq1}\SP$ is $1$-connective, so is built under colimits from $\Sigma \SP$.
\end{proof}

The goal of this section is to prove \Cref{thm:cnbousfieldclassmain}. We begin by reformulating connective Bousfield classes in terms of localizations:

%
%

\begin{definition}
	Let $C$ be a presentable category and let $f \in C$ be a morphism. We define $L_f:C \to L_fC$ to be the presentable localization that inverts the morphism $f$. In the case that $f$ is a map $X \to *$, we write $P_X:C \to P_XC$. 
\end{definition}

\begin{lemma}\label{lemma:kerlocalization}
	Let $C$ be a presentable stable 
	category. Then the collection $\{Z|P_XZ=0\}$ is the smallest subcategory generated under colimits and extensions by $X$. 
\end{lemma}

\begin{proof}
	It is easy to see that the collection of objects with $P_X(Y)=*$ is closed under extensions and colimits and contains $X$ by assumption. To prove the other inclusion, we recall that $P_X(Y)$ can be computed using a small object argument, by transfinitely pushing out along homotopy classes of maps from $\Sigma^nX\to *$ for $n\geq 0$. By taking the fiber of the map from $Y$ at each step of this process, we see the claim.
\end{proof}

%

We use $\Sp_{\geq0,p}$ to denote the category of $p$-complete connective spectra. Our goal is to study the localization $P_{\Sigma^{\infty}BC_p}\Sp_{\geq0,p}$.

%

\begin{lemma}\label{lemma:finitegroup}
	In $P_{\Sigma^{\infty}BC_p}\Sp_{\geq0,p}$, the map $\Sigma^{\infty}_+BG \to \Sigma^{\infty}_+*$ is an equivalence for every finite group $G$.
\end{lemma}

\begin{proof}
	The transfer realizes $\Sigma^{\infty}BG$ as a retract of $\Sigma^{\infty}G_p$, where $G_p$ is a $p$-Sylow subgroup of $G$, so we may assume that $G$ is a $p$-group. $G$ can be written as an extension $H \to G \to \ZZ/p\ZZ$. By induction on $|G|$ we may assume that the map $\Sigma^{\infty}_+BH \to \Sigma^{\infty}_+*$ is an equivalence. Taking orbits by the action of $\ZZ/p\ZZ$ coming from the extension, we learn that $\Sigma^{\infty}_+BG \to \Sigma^{\infty}_+\ZZ/p\ZZ$ is an equivalence, which composing with the isomorphism $\Sigma^{\infty}_+\ZZ/p\ZZ \to \Sigma^{\infty}_+*$ gives the inductive step.
\end{proof}

Note that kernel of the localization $\Sp_{\geq0,p} \to P_{\Sigma^{\infty}BC_p}\Sp_{\geq0,p}$ is a quotient by a tensor ideal. Thus this localization is symmetric monoidal.

A key observation is that the free $\EE_\infty$-ring on a class in degree $0$ in $P_{\Sigma^{\infty}BC_p}\Sp_{\geq0,p}$ is a polynomial ring.
\begin{lemma}\label{lemma:freestrict}
	In $P_{\Sigma^{\infty}BC_p}\Sp_{\geq0,p}$, the map $\SP\{x\} \to \Sigma^{\infty}_+\NN$ sending $x$ to $1 \in \NN$ is an equivalence.
\end{lemma}

\begin{proof}
	We apply \Cref{lemma:finitegroup} to the case $G = \Sigma_n$ to obtain the lemma. 
\end{proof}

%

We are now ready to finish the proof of \Cref{thm:cnbousfieldclassmain}.
\begin{proof}[Proof of \Cref{thm:cnbousfieldclassmain}]
	Using a homology decomposition, we see that $\Sigma^{\infty}BC_p$ is in the subcategory generated under colimits by $\Sigma S/p$, so we claim that in $P_{\Sigma^{\infty}BC_p}\Sp_{\geq0,p}$, $\Sigma \SP/p=0$. To see that this claim finishes the proof, by \Cref{lemma:kerlocalization}, this implies that $\Sigma \SP/p$ is in the subcategory of $p$-complete spectra generated under colimits and extensions by $\Sigma^{\infty}BC_p$. But we may embed $\Sp_{\geq0,p}$ into $\Sp$ by taking the fiber of the rationalization map. This functor preserves colimits, and sends $\Sigma \SP/p$ and $\Sigma^{\infty}BC_p$ to itself, so we learn the desired conclusion.
	
	We turn to proving the claim. By taking the bar construction over the augmentation of the isomorphism of \Cref{lemma:freestrict} twice, we learn that there is an isomorphism of augmented $\EE_{\infty}$-algebras
	$$\SP\{x_2\}\simeq  \Sigma^{\infty}_+\CC\PP^{\infty}$$ where $x_2$ is a class in degree $2$.
	
	But $\Sigma^{\infty}\CC\PP^{\infty} = \Sigma^{\infty}B\QQ_p/\ZZ_p $ since we are working $p$-adically, and the latter is $0$ by \Cref{lemma:finitegroup} since $\QQ_p/\ZZ_p$ is a filtered colimit of finite groups. Thus the augmentation ideal of $\SP\{x_2\}$ is zero, so in particular $\Sigma^2\SP = 0$, since this is a retract of the augmentation ideal. Thus the category is $1$-truncated, so we have $0= \Sigma^{\infty}BC_p = \tau_{\leq1}\Sigma^{\infty}BC_p = \tau_{\leq 1}\Sigma \SP/p =\Sigma \SP/p$.

\end{proof}

We now give an alternative proof of \Cref{thm:cnbousfieldclassmain} using the symmetric power filtration on $\ZZ$. We thank Gijs Heuts for discussions that led to this alternative proof.

\begin{proof}[Second proof of \Cref{thm:cnbousfieldclassmain}]\label{proof:number2}
	By \cite[Theorem A]{mitchell1984symmetric} and \cite[Proposition 5.15]{mitchell1984symmetric}, if we consider the filtration $\mathrm{Sp}^{p^n}\SP_{(p)}$ of $\ZZ_{(p)}$, then the associated graded $\mathrm{Sp}^{p^n}\SP_{(p)}/\mathrm{Sp}^{p^{n-1}}\SP_{(p)}$ for $n\geq 1$ is up to a suspension a summand of $\Sigma^{\infty}(BC_p)^n$. These associated graded pieces are connected and their homology groups are $p$-power torsion and finitely generated, so by \Cref{lemma:cnsuspension} and \Cref{lemma:kerlocalization}, it follows that they are killed by $P_{\Sigma^{\infty}BC_p}$. It then follows that $\SP_{(p)} \to \ZZ_{(p)}$ is an equivalence after applying $P_{\Sigma^{\infty}BC_p}$, and so $\SP/p \to \ZZ/p$ is too. But $P_{\Sigma^{\infty}BC_p}(\Sigma \ZZ/p)=0$, since $\Sigma^{\infty}BC_p\otimes \ZZ$ has $\Sigma \ZZ/p$ as a summand. Thus $P_{\Sigma^{\infty}BC_p}(\Sigma \SP/p) = P_{\Sigma^{\infty}BC_p}(\Sigma \ZZ/p) = 0$, which implies the nontrivial inclusion in the theorem using \Cref{lemma:kerlocalization}.
\end{proof}

\section{Unstable Bousfield classes}

Recall that the \textit{unstable Bousfield class} of a space $X$, denoted $\langle X\rangle^{\un}$, is the kernel of the localization $P_X$.

We recall the following result of Bousfield relating unstable and stable Bousfield classes:

\begin{proposition}[{\cite[Proposition 3.1]{bousfield1996unstable}}]\label{prop:cnunstablebousclasss}
	Let $X$ be a pointed space and $W$ be a connective spectrum. Then $\langle X\rangle^{\un}\supset \langle \Omega^{\infty}W\rangle^{\un}$ iff $\langle\Sigma^{\infty} X\rangle^{s} \supset \langle W\rangle^{s}$. In particular, $\langle X\rangle^{\un}\supset \langle \Omega^{\infty}\Sigma^{\infty}X\rangle^{\un}$ and $\langle \Sigma^{\infty}\Omega^{\infty}W\rangle^{s} \supset \langle W\rangle^{s}$.
\end{proposition}

\begin{theorem}\label{cor:unstablebousclass}
	There is an equality of unstable Bousfield classes $\langle Q(S^1/p)\rangle^{\un} = \langle BC_p\rangle^{\un}$. Moreover this unstable Bousfield class contains $\Omega^{\infty}X$ for any connected spectrum $X$ with $X[p^{-1}]=0$.
\end{theorem}

\begin{proof}
	Taking $X = BC_p$ and $W = S^1/p$, \Cref{thm:cnbousfieldclassmain} along with \Cref{prop:cnunstablebousclasss} implies that $\langle BC_p\rangle^{\un} \supset \langle Q(S^1/p)\rangle^{\un}$. For the other inclusion, we can take $X = Q(S^1/p)$ and $W = \Sigma \FF_p$, and apply \Cref{prop:cnunstablebousclasss} again.
	
	To show that $\Omega^{\infty}X$ is contained in $\langle \Omega^{\infty}X\rangle$ for any connected spectrum with $X[p^{-1}]$, we note that such $X$ are contained in $\langle \Sigma \SP/p\rangle^{s}$, and apply \Cref{prop:cnunstablebousclasss}.
%
\end{proof}

The following localizations will be useful in studying the Sullivan conjecture.
\begin{definition}
	We let $P_{\vee_{p\in \PP}BC_p}$ denote the presentable localization of spaces obtained from inverting the maps $BC_p \to *$, and let $L$ denote the presentable localization obtained by inverting the maps $BC_p \to *$ and $\FF_p$-homology equivalences.
\end{definition}

Thess localizations are relevant because of the following lemma:

\begin{lemma}\label{lemma:pcompletefinitelocal}
	If $X$ is a finite dimensional $\FF_p$-complete nilpotent space, then $X$ is $L$-local. If $X$ is a the homotopy type of a finite dimensional CW-complex, then $X$ is $P_{\vee_{p\in \PP}BC_p}$-local.
\end{lemma}

\begin{proof}
	For the first statement, \Cref{thm:miller} implies that it is local with respect to $BC_p \to *$. The second statement follows from \cite[Theorem A]{miller1984sullivan}.
\end{proof}

\begin{lemma}\label{lemma:padictoruslocal}
	For any $p$-complete torsion free abelian group $M$, $BM$ is $L$-local.
\end{lemma}

\begin{proof}
	$BM$ is $\FF_p$-complete since it is generated under limits by $B\FF_p$ and $B^2\FF_p$, which are $\FF_p$-complete. It is local with respect to $BC_p \to *$ because it is torsion free.
\end{proof}

\begin{theorem}\label{thm:main}
	Let $X$ be an infinite loop space. Then the $L$-localization of $X$ is given by the map $X \to \coprod_{\pi_0X}B((\pi_1X/\mathrm{tors})^{\wedge}_p)$ which is a bijection on components and on each component induces the map $\pi_1X \to (\pi_1X/\mathrm{tors})^{\wedge}_p$ on fundamental groups. In particular, this induces an equivalence on mapping spaces into any finite dimensional $p$-complete nilpotent space.
\end{theorem}

\begin{proof}
	Since $L$-localization can be computed component-wise, we may assume $X = \Omega^{\infty}x$, for $x$ a connected spectrum.
	We first claim that $X \to B(\pi_1X/\mathrm{tors})$ is an equivalence after applying $L$. To do this, it suffices to show that $L$ applied to the fiber $F$ is contractible. We may replace $F$ with $\fib(F \to F[\frac 1 p])$, since the comparison map is an $\FF_p$-homology equivalence. But this fiber is in $\langle BC_p\rangle^{\mathrm{{un}}}$ by \Cref{cor:unstablebousclass}, so $L$ applied to it is contractible.
	
	
	The map $B(\pi_1X/\mathrm{tors}) \to B(\pi_1X/\mathrm{tors})_p$ is an $\FF_p$-homology equivalence, and so since $B(\pi_1X/\mathrm{tors})_p$ is $L$-local by \Cref{lemma:padictoruslocal}, we are done.
\end{proof}

We next prove the corollary from the introduction:

\begin{proof}[Proof of \Cref{cor:action}]
	An action of an infinite loop space $X$ on a space $Y$ is an $\EE_1$-algebra map $X \to \Map(Y,Y)$. Since $Y$ is $L$-local, $\Map(Y,Y)$ is also $L$-local, since it is obtained via limits from $Y$. Thus by \Cref{thm:main}, and the fact that presentable localizations of spaces are symmetric monoidal with respect to the product, we get a factorization of the action through $\coprod_{\pi_0X}B((\pi_1X/\mathrm{tors})^{\wedge}_p)$, which is an extension of $\pi_0X$ by the $p$-adic torus $B((\pi_1X/\mathrm{tors})^{\wedge}_p)$.
\end{proof}

We next extract an integral version of our theorem:

\begin{theorem}\label{thm:integralspecsull}
	Let $X$ be an infinite loop space with $\pi_iX\otimes \QQ=0$ for $i>0$. Then the map $X \to \pi_0X$ is the $P_{\vee_{p\in \PP}BC_p}$-localization map, and hence induces an equivalence on mapping spaces into $Y$ for each finite dimensional $Y$.
\end{theorem}

\begin{proof}
	Since $P_{\vee_{p\in \PP}BC_p}$ preserves connected components, we may assume WLOG that $X = \Omega^{\infty}x$ where $x$ is rationally trivial and connected. Then by using \Cref{prop:cnunstablebousclasss} and \Cref{thm:cnbousfieldclassmain}, the result follows since $x \in \langle \oplus_{p \in \PP}\Sigma \SP/p\rangle^s$.
\end{proof}

The proof of \Cref{cor:action} allows one to obtain the following integral version:
\begin{corollary}\label{cor:intaction}
	An action of an infinite loop space $X$ with $\pi_iX\otimes \QQ=0$ for $i>0$ on a finite dimensional space $Y$ factors through an action of $\pi_0X$.
\end{corollary}

	\nocite{}
	\printbibliography
	\Addresses
\end{document}

%% file: template.tex
\usepackage{geometry}
\usepackage{graphicx}
\usepackage{todonotes}
\usepackage{amsmath}
\usepackage{amsfonts}
\usepackage{amssymb}
\usepackage{tikz-cd}
\usepackage{comment}
\usepackage{bbm}
\usepackage{stmaryrd}
\usepackage{fullpage}
\usepackage{epstopdf}
\usepackage{amsthm}
\usepackage{mathrsfs} 
\usepackage{tikz}
\usepackage{dsfont}
\usetikzlibrary{matrix}
\usepackage{mathtools}
\usepackage[hidelinks]{hyperref}
\usepackage{cleveref}
\usepackage{xcolor}
\usepackage{enumitem}
\usepackage{tensor}
\usepackage[utf8]{inputenc}
\usepackage{csquotes}
\usepackage[english]{babel}
\usepackage{mymacros}

\hypersetup{
	colorlinks,
	linkcolor={red!50!black},
	citecolor={blue!50!black},
	urlcolor={blue!80!black}
}

\tikzset{
	rot90/.style={anchor=south, rotate=90, inner sep=.5mm}
}
\tikzset{
	rot45/.style={anchor=south, rotate=-45, inner sep=.5mm}
}

\newtheorem{theorem}{Theorem}[section]

\newtheorem{lemma}[theorem]{Lemma}

\newtheorem{proposition}[theorem]{Proposition}
\newtheorem{corollary}[theorem]{Corollary}
\newtheorem{conjecture}[theorem]{Conjecture}

\theoremstyle{definition}
\newtheorem{definition}[theorem]{Definition}
\newtheorem{remark}[theorem]{Remark}